\documentclass[11pt, a4paper]{article}

\usepackage{amsmath}
\usepackage{amsfonts}
\usepackage{dsfont}
\usepackage{mathrsfs}
\usepackage{bbm}
\usepackage{latexsym}
\usepackage{graphicx}
\usepackage{amssymb}
\usepackage{float}
\usepackage{epsfig}
\usepackage{epstopdf}
\usepackage{color}

\newtheorem{theorem}{Theorem}[section]
\newtheorem{lemma}[theorem]{Lemma}
\newtheorem{corollary}[theorem]{Corollary}
\newtheorem{proposition}[theorem]{Proposition}

\newtheorem{conjecture}[theorem]{Conjecture}

\title{{\bf On the $A_{\alpha}$-spectra of graphs}}
\author{ Huiqiu Lin$^{a}$\thanks{Corresponding author. Email:~huiqiulin@126.com (H. Lin).},~~Jie Xue$^{b}$,~~Jinlong Shu$^{b}$
\\
{\footnotesize $^a$Department of Mathematics, East China University of Science and Technology, Shanghai, PR China}\\
{\footnotesize $^b$Department of Computer Science and Technology, East China Normal University, Shanghai, PR China }}

\date{}
\begin{document}
\maketitle

\begin{abstract}
Let $G$ be a graph with adjacency matrix $A(G)$ and let $D(G)$
be the diagonal matrix of the degrees of $G$. For any real $\alpha\in [0,1]$, Nikiforov \cite{VN1} defined the matrix $A_{\alpha}(G)$ as
$$A_{\alpha}(G)=\alpha D(G)+(1-\alpha)A(G).$$
In this paper, we give some results on the eigenvalues of $A_{\alpha}(G)$ for $\alpha>1/2$. In particular, we characterize the graphs with $\lambda_k(A_{\alpha}(G))=\alpha n-1$ for $2\leq k\leq n$. Moreover, we show that $\lambda_n(A_{\alpha}({G}))\geq 2\alpha-1$ if $G$ contains no isolated vertices.


\end{abstract}

\section{Introduction}
All graphs considered here are simple and undirected. Let $G$ be a graph with adjacency matrix $A(G)$, and let $D(G)$ be the diagonal matrix of the degrees of $G$. For any real $\alpha\in [0,1]$, Nikiforov \cite{VN1} defined the matrix $A_{\alpha}(G)$ as
$$A_{\alpha}(G)=\alpha D(G)+(1-\alpha)A(G).$$
It is clear that $A_{\alpha}(G)$ is the adjacency matrix if $\alpha=0$, and $A_{\alpha}(G)$ is essentially equivalent to signless Laplacain matrix if $\alpha=1/2$.

When $\alpha\geq 1/2$, Nikiforov \cite{VN1} proved that $\lambda_{2}(A_{\alpha}(G))\leq \alpha n-1$. This implies that $\lambda_k(A_{\alpha}(G))\leq\lambda_{2}(A_{\alpha}(G))\leq\alpha n-1$ for $k\geq 2$.
de Lima and Nikiforov \cite{LN} showed that $\lambda_k(A_{\frac{1}{2}}(G))=\frac{1}{2}n-1 \ \mbox{for $k\geq2$}$ if and only if $G$ has
either $k$ balanced bipartite components or $k + 1$ bipartite components. Let $\partial_1(G)\geq\cdots\geq\partial_n(G)$ denote the distance signless Laplacian spectrum.
Lin and Das \cite{LD} proved $\partial_n(G) =n-2$ if and only if $G^c$ contains either
a balanced bipartite graph or at least two bipartite components. In this paper, we characterize the graphs with $\lambda_k(A_{\alpha}(G))=\alpha n-1$ for $k\geq 2$ when $\alpha> 1/2$,
these graphs are not the same as those for $\alpha=1/2$.

\begin{theorem}\label{thm2}
Let $G$ be a graph of order $n$. If $\alpha> 1/2$, then
$\lambda_k(A_{\alpha}(G))= \alpha n-1 \ \mbox{for $k\geq2$}$ if and only if $G$ has $k$ vertices of degree $n-1$.
\end{theorem}

It is well-known that $\lambda_n(A_{\frac{1}{2}}(G))=0$ if $G$ is bipartite. So it is interesting to give a lower bounds of $\lambda_n(A_{\frac{1}{2}}(G))$ when $G$ is non-bipartite.
In 2007, Cvetkovi\'{c}, Rowlinson and Simi\'c \cite{CR} proposed the following conjecture:

\begin{conjecture}
Let $G$ be a non-bipartite graph with order $n$. Then $$\lambda_n(A_{\frac{1}{2}}(G))\geq \lambda_n(A_{\frac{1}{2}}(G^*)),$$
equality holds if and only if $G\cong G^*$,
where $G^*$ is the unicyclic graph obtained from a triangle by attaching a
path at one of its end vertices.
\end{conjecture}

One year later, the conjecture was confirmed by Cardoso, Cvetkovi\'{c}, Rowlinson and Simi\'c \cite{CC}. In \cite{LO}, de Lima, Oliveira, de Abreu and Nikiforov proved that $\lambda_n(A_{\frac{1}{2}}(G))\leq \frac{m}{n}-\frac{1}{2}$ and in the same paper, they proposed a conjecture on the lower bound of $\lambda_n(A_{\frac{1}{2}}(G))$.
\begin{conjecture}
Let $G$ be a graph with order $n$ and size $m$. Then $$\lambda_n(A_{\frac{1}{2}}(G))\geq \frac{m}{n-1}-\frac{n-2}{2}.$$
\end{conjecture}

Guo, Chen and Yu \cite{GCY} proved a stronger result, $\lambda_n(A_{\frac{1}{2}}(G))\geq \frac{m}{n-2}-\frac{n-1}{2}$. For $\alpha>1/2$,
Nikiforov (\cite{VN1}, Proposition 7) showed that
$A_{\alpha}(G)$ is positive definite if $G$ has no isolated vertices, that is, $\lambda_n(A_{\alpha}({G}))>0$. In the following, we give a lower bound on $\lambda_n(A_{\alpha}({G}))$ when $G$ has no isolated vertices.
\begin{theorem}\label{thm4}
Let $G$ be a graph of order $n$. If $G$ has no isolated vertices and $\frac{1}{2}<\alpha<1$, then
$$\lambda_n(A_{\alpha}({G}))\geq 2\alpha-1,$$
the equality holds if and only if $G$ has a component isomorphic to $K_{2}$.
\end{theorem}

The rest of the paper is organized as follows. In Section 2, we study the $k$-th largest $A_{\alpha}$-eigenvalue, and give the proof of Theorem \ref{thm2}. In Section 3, we present the proof of Theorem \ref{thm4}, and we also give some results on the smallest $A_{\alpha}$-eigenvalue.

\section{The $k$-th largest eigenvalue of $A_{\alpha}(G)$}

Let $G$ be a graph with vertex set $V(G)$ and edge set $E(G)$. The \emph{degree} of a vertex $v\in V(G)$ is denoted by $d_{G}(v)$. We use $N(v)$ to denote the set of vertices of $G$ which are adjacent to $v$. We first list some fundamental properties of the matrix $A_{\alpha}(G)$,
which can be found in \cite{VN1}.  Let $x$ be a vector on the vertices of $G$. We use $x(v)$ to denote the entry of $x$ corresponding to the vertex $v\in V(G)$. Suppose that $\lambda$ is an eigenvalue of $A_{\alpha}(G)$. If $x$ is an eigenvector of $A_{\alpha}(G)$ with respect to $\lambda$, then the eigenequations for $A_{\alpha}(G)$ can be written as
$$\lambda x(v)=\alpha d_{G}(v)x(v)+(1-\alpha)\sum_{u\in N(v)}x(u),$$
where $v\in V(G)$. The quadratic form $\langle A_{\alpha}(G)x,x\rangle$ can be represented in several ways:
\begin{eqnarray*}
\langle A_{\alpha}(G)x,x\rangle&=&\sum_{uv\in E(G)}(\alpha x(u)^{2}+2(1-\alpha)x(u)x(v)+\alpha x(v)^{2}),\\
\langle A_{\alpha}(G)x,x\rangle&=&(2\alpha-1)\sum_{u\in V(G)}x(u)^{2}d_{G}(u)+(1-\alpha)\sum_{uv\in E(G)}(x(u)+x(v))^{2},\\
\langle A_{\alpha}(G)x,x\rangle&=&\alpha\sum_{u\in V(G)}x(u)^{2}d_{G}(u)+2(1-\alpha)\sum_{uv\in E(G)}x(u)x(v).
\end{eqnarray*}

For convenience we state below the complete theorem of Weyl and So.

\begin{theorem}\rm{(\cite{SW})}\label{lem2.1} Let $A$ and $B$ be $n\times n$ Hermitian matrices and $C = A+B$. Then
$$\lambda_i(C) \leq \lambda_j(A) + \lambda_{i-j+1}(B) (n \geq i \geq j \geq 1),$$
$$\lambda_i(C) \geq \lambda_j(A) + \lambda_{i-j+n}(B) (1 \leq i \leq j \leq n).$$
In either of these inequalities equality holds if and only if there exists a nonzero $n$-vector that is an eigenvector to each of the three
involved eigenvalues.
\end{theorem}

\begin{proposition}\rm{(\cite{LN})}\label{prop}
  Let $2\leq k<n$ and $A$ and $B$ be self-adjoint operators of order $n$. If for every $s=2,\ldots, k$,
  $$\lambda_{s}(A)+\lambda_{n}(B)=\lambda_{s}(A+B),$$
  then there exist $k-1$ nonzero orthogonal $n$-vectors $x_{1},\ldots,x_{k-1}$ such that
  $$Ax_{s-1}=\lambda_{s}(A)x_{s-1}, Bx_{s-1}=\lambda_{n}(B)x_{s-1}~ \text{and}~(A+B)x_{s-1}=\lambda_{s}(A+B)x_{s-1}$$
  for every $s=2,\ldots,k.$
\end{proposition}

Let $\alpha\geq 1/2$.  We use $G-e$ to denote the graph obtained from $G$ by deleting an edge $e\in E(G)$. Thus, $A_{\alpha}(G)=A_{\alpha}(G-e)+N$, where $N$ is the $A_{\alpha}$-matrix of a graph containing only one edge. In \cite{VN1}, it was proved that the $A_{\alpha}$-matrix is positive semi-definite if $\alpha\geq 1/2$. Hence, the matrices $A_{\alpha}(G)$, $A_{\alpha}(G-e)$ and $N$ are all positive semi-definite.
By using \emph{monotonicity theorem} (see \cite{HR}, Corollary 4.3.3), the following result holds.

\begin{proposition}\label{prop1}
  Let $G$ be a graph of order $n$. If $e\in E(G)$ and $\alpha\geq 1/2$, then
$$\lambda_i(A_{\alpha}(G))\geq\lambda_i(A_{\alpha}(G-e))$$
for $1\leq i\leq n$.
\end{proposition}

If $S\subseteq V(G)$, then we use $G[S]$ to denote the subgraph of $G$ induced by $S$.

\begin{proposition}\label{prop4.1}
Let $G$ be a graph with $A_{\alpha}(G)$ and $0\leq \alpha \leq1$. Let $S\subseteq V(G)$ and $|S|=k$.
Suppose that $d_{G}(w)=d$ for each vertex $w\in S$, and $N(v)\backslash \{u\}=N(u)\backslash \{v\}$ for any two vertices $u,v\in S$. Then we have the following statements.

\noindent $(1)$ If $G[S]$ is a clique, then $(d+1)\alpha-1$ is an eigenvalue of $A_{\alpha}(G)$ with multiplicity at least $k-1$.

\noindent $(2)$ If $G[S]$ is an independent set, then $d\alpha$ is an eigenvalue of $A_{\alpha}(G)$ with multiplicity at least $k-1$.
\end{proposition}
\begin{proof}
Let $S=\{v_{1},v_{2},\ldots,v_{k}\}$. Clearly, $d_{G}(v_{1})=\cdots=d_{G}(v_{k})=d$. Let $x_{1},x_{2},\ldots,x_{k-1}$ be vectors such that
$$\left\{
\begin{array}{l}
x_{i}(v_{1})=1,\\
x_{i}(v_{i+1})=-1,\\
x_{i}(v)=0  ~~\text{if}~~ v\in V(G)\backslash \{v_{1},v_{i+1}\},
\end{array}
\right.$$
 for $i=1,\ldots,k-1$. Suppose that $G[S]$ is a clique. One can easily obtain that
$$A_{\alpha}(G)x_i=((d_{G}(v_{1})+1)\alpha-1,0,\ldots,0,1-(d_{G}(v_{i})+1)\alpha,0,\ldots,0)^{'}=((d+1)\alpha-1)x_i,$$
for $i=1,\ldots,k-1$. Hence, $(d+1)\alpha-1$ is an eigenvalue of $A_{\alpha}(G)$ and $x_1,\ldots,x_{k-1}$ are eigenvectors of $A_{\alpha}(G)$ corresponding to $(d+1)\alpha-1$. Moreover, since $x_1,\ldots,x_{k-1}$ are linearly independent,
the multiplicity of $(d+1)\alpha-1$ is at least $k-1$.

When $G[S]$ is an independent set, we have
$$A_{\alpha}(G)x_i=(d_{G}(v_{1})\alpha,0,\ldots,0,-d_{G}(v_{i})\alpha,0,\ldots,0)^{'}=d\alpha x_i,$$
for $i=1,2,\ldots,k-1$. Since $x_1,\ldots,x_{k-1}$ are linearly independent, it follows that $d\alpha$ is an eigenvalue of $A_{\alpha}(G)$ with multiplicity at least $k-1$.
Thus we complete the proof. \hspace*{\fill}$\Box$
\end{proof}

The graph $K_s\vee (n-s)K_1$ is called \emph{complete split graph}, denoted by $CS_{s,n-s}$. According to Proposition \ref{prop4.1} and [\cite{VN1}, Proposition 37], we determine all $A_{\alpha}$-eigenvalues of $CS_{s,n-s}$.
\begin{corollary}\label{cor1}
The $A_{\alpha}$-spectrum of $CS_{s,n-s}$ contains $n\alpha -1$ with multiplicity $s-1$, $s\alpha$ with multiplicity $n-s-1$
and the remaining two $A_{\alpha}$-eigenvalues are $$\frac{n\alpha+s-1\pm \sqrt{(n\alpha+s-1)^2-4s\alpha(s-1)+4s(n-s)(1-2\alpha)}}{2}.$$
\end{corollary}

At the end of this section, we give the proof of Theorem \ref{thm2}.

\noindent{\bf Proof of Theorem \ref{thm2}.} If $G$ has at least $k$ vertices of degree $n-1$, then it follows from Proposition \ref{prop4.1} that $n\alpha-1$ is an eigenvalue of $A_{\alpha}(G)$ with multiplicity at least $k-1$. Since $CS_{k,n-k}$ is a subgraph of $G$, we have $\lambda_1(A_{\alpha}(G))\geq \lambda_1(A_{\alpha}(CS_{k,n-k}))>n\alpha -1$, and so $\lambda_k(A_{\alpha}(G))=\alpha n-1.$ On the other hand, we assume that $\lambda_k(A_{\alpha}(G))=\alpha n-1$. Hence,
$\lambda_{2}(A_{\alpha}(G))=\cdots=\lambda_{k}(A_{\alpha}(G))=\alpha n-1.$
 If $k=n$, then $G\cong K_{n}$ (otherwise, $\lambda_n(A_{\alpha}(G))\leq \lambda_n(A_{\alpha}(CS_{n-2,2})<n\alpha -1)$). Suppose $k\leq n-1$. Let $G^{c}$ be the complement of $G$. Note that $$A_{\alpha}(G)+A_{\alpha}(G^c)=A_{\alpha}(K_n).$$
 By Theorem \ref{lem2.1}, we have $$\lambda_i(A_{\alpha}(G))+\lambda_n(A_{\alpha}(G^c))\leq \lambda_i(A_{\alpha}(K_n)) \ \mbox{for $2\leq i\leq n$}.$$
 Since $\lambda_k(A_{\alpha}(G))=\alpha n-1$ and $\lambda_k(A_{\alpha}(K_{n}))=\alpha n-1$, it follows that $\lambda_n(A_{\alpha}(G^c))=0$. Hence,
 $$\lambda_i(A_{\alpha}(G))+\lambda_n(A_{\alpha}(G^c))=\lambda_i(A_{\alpha}(K_n))$$
 for every $i=2,\ldots,k$. According to Proposition \ref{prop}, we can see that there exist $k-1$ nonzero orthogonal $n$-vectors $x_{1},\ldots,x_{k-1}$ such that
  $$A_{\alpha}(G)x_{i-1}=\lambda_{i}(A_{\alpha}(G))x_{i-1}, A_{\alpha}(G^c)x_{i-1}=\lambda_{n}(A_{\alpha}(G^c))x_{i-1}~ \text{and}~A_{\alpha}(K_n)x_{i-1}=\lambda_{i}(A_{\alpha}(K_n))x_{i-1}$$
  for every $i=2,\ldots,k.$ Let $V^{*}=\{v\in V(G): d_{G}(v)=n-1\}$. Since $\lambda_n(A_{\alpha}(G^c))=0$, $G^{c}$ contains isolated vertex, and so $V^{*}\neq \emptyset$. Suppose $V^{*}=\{v_{1},v_{2},\ldots,v_{t}\}$. In the following, we only need to show $t\geq k$. For $1\leq i\leq k-1$, since $A_{\alpha}(G^c)x_{i}=\lambda_{n}(A_{\alpha}(G^c))x_{i}$ and $\lambda_n(A_{\alpha}(G^c))=0$, we have
  $$0=\langle A_{\alpha}(G^c)x_{i}, x_{i}\rangle=(2\alpha-1)\sum_{u\in V(G^c)}d_{G^c}(u)x_{i}(u)^2+(1-\alpha)\sum_{uv\in E(G^c)}(x_{i}(u)+x_{i}(v))^2.$$
 Since $\alpha>1/2$, one can see that $x_{i}(u)=0$ for all $u\in V(G)\backslash V^{*}.$ Moreover, since $x_{i}$ is also an eigenvector of $\lambda_{i+1}(A_{\alpha}(K_n))=\alpha n-1$, we have $\textrm{1}^{'}x_{i}=0$ and hence $\sum_{v\in V^{*}}x_{i}(v)=0$, that is, $x_{i}(v_1)=-(x_{i}(v_2)+\cdots+x_{i}(v_t))$. Suppose that $y_{1},\ldots,y_{t-1}$ are $n$-vectors such that $y_{i}(v_{1})=-1$, $y_{i}(v_{i+1})=1$ and $y_{i}(v)=0$ if $v\in V(G)\backslash \{v_{1},v_{i+1}\}$, for $i=1,\ldots,t-1$. Thus, it is easy to see that
 $$x_{i}=x_{i}(v_{2})y_{1}+x_{i}(v_{3})y_{2}+\cdots+x_{i}(v_{t})y_{t-1}$$
 for $i=1,\ldots,k-1$. Therefore, $x_{1},\ldots,x_{k-1}$ can be represented by using $y_{1},\ldots,y_{t-1}$. Since these two vector families are both linearly independent, it follows that $t\geq k$, thus we complete the proof.\hspace*{\fill}$\Box$

\section{The smallest eigenvalue of $A_{\alpha}(G)$}

In this section, we study the smallest eigenvalue of $A_{\alpha}(G)$. First, we will give the proof of Theorem \ref{thm4}. Before proceeding, the following lemma is needed.
\begin{lemma}\label{le1+}
Let $T$ be a tree of order $n\geq 2$. If  $\frac{1}{2}<\alpha<1$, then $\lambda_n(A_{\alpha}({T}))\geq 2\alpha-1$ with equality if and only if $T\cong K_{2}$.
\end{lemma}

\begin{proof}
Clearly, $\lambda_{2}(A_{\alpha}(K_2))=2\alpha-1$. It suffices to show that $\lambda_n(A_{\alpha}({T}))> 2\alpha-1$ if $n\geq 3$. We prove this by induction on the order $n$. Recall that the smallest eigenvalue of star $K_{1,n-1}$ is $$\lambda_n(A_{\alpha}(K_{1,n-1}))=\frac{1}{2}(\alpha n-\sqrt{\alpha^2n^2+4(n-1)(1-2\alpha)}).$$ It is easy to check that $\frac{1}{2}(\alpha n-\sqrt{\alpha^2n^2+4(n-1)(1-2\alpha)})>2\alpha-1,$ if $\frac{1}{2}<\alpha<1$ and $n\geq 3$. Hence, the result follows when $T$ is a star. So, in the following we may assume that $T\ncong K_{1,n-1}$. Then there exists a non-pendent edge $e$ such that $T-e=T_{1}\cup T_{2}$ and $|V(T_{1})|\geq |V(T_{2})|\geq 2$.

\noindent{\bf{Case 1.}} $|V(T_{1})|\geq 3$ and $|V(T_{2})|\geq 3$.

By the induction hypothesis, we have $\lambda_{|V(T_{1})|}(A_{\alpha}(T_{1}))>2\alpha-1$ and $\lambda_{|V(T_{2})|}(A_{\alpha}(T_{2}))>2\alpha-1$. Note that $\lambda_{n}(A_{\alpha}(T-e))=\min\{\lambda_{|V(T_{1})|}(A_{\alpha}(T_{1})),\lambda_{|V(T_{2})|}(A_{\alpha}(T_{2}))\}$. According to Proposition \ref{prop1}, it follows that
$$\lambda_{n}(A_{\alpha}(T))\geq \lambda_{n}(A_{\alpha}(T-e))=\min\{\lambda_{|V(T_{1})|}(A_{\alpha}(T_{1})),\lambda_{|V(T_{2})|}(A_{\alpha}(T_{2}))\}>2\alpha-1,$$
as required.

\noindent{\bf{Case 2.}} $|V(T_{1})|= 2$ and $|V(T_{2})|= 2$.

Thus $T\cong P_{4}$. By a simple calculation, we have $$\lambda_{4}(A_{\alpha}(P_{4}))=\min\{\alpha+\frac{1}{2}-\frac{1}{2}\sqrt{4\alpha^2-8\alpha+5}, 2\alpha-\frac{1}{2}-\frac{1}{2}\sqrt{8\alpha^2-12\alpha+5}\}.$$
It is easy to check that both $\alpha+\frac{1}{2}-\frac{1}{2}\sqrt{4\alpha^2-8\alpha+5}>2\alpha-1$ and $2\alpha-\frac{1}{2}-\frac{1}{2}\sqrt{8\alpha^2-12\alpha+5}>2\alpha-1$, hence $\lambda_{4}(A_{\alpha}(P_{4}))>2\alpha-1$.

\noindent{\bf{Case 3.}} $|V(T_{1})|\geq 3$ and $|V(T_{2})|= 2$.

That is $T-e=T_{1}\cup K_{2}$. We may assume that $e=uv$ and $K_{2}=vw$. Hence
$$A_{\alpha}(T)=A_{\alpha}(T_{1}\cup K_{2})+M$$
where $M=A_{\alpha}(K_{2}\cup (n-2)K_1)$. The eigenvalues of $M$ are $1,2\alpha-1,0,\ldots,0$. According to Theorem \ref{lem2.1}, we have
$$\lambda_{n}(A_{\alpha}(T))\geq\lambda_{n}(A_{\alpha}(T_{1}\cup K_{2})) +\lambda_{n}(M).$$
By the induction hypothesis, we have $\lambda_{|V(T_{1})|}(A_{\alpha}(T_{1}))>2\alpha-1$, then $\lambda_{n}(A_{\alpha}(T_{1}\cup K_{2}))=\lambda_{2}(A_{\alpha}(K_{2}))=2\alpha-1$. And since $\lambda_{n}(M)=0$, therefore
$$\lambda_{n}(A_{\alpha}(T))\geq\lambda_{n}(A_{\alpha}(T_{1}\cup K_{2})) +\lambda_{n}(M)=2\alpha-1.$$
If $\lambda_{n}(A_{\alpha}(T))=2\alpha-1$, then it follows from Theorem \ref{lem2.1} that $\lambda_{n}(A_{\alpha}(T))$, $\lambda_{n}(A_{\alpha}(T_{1}\cup K_{2}))$ and $\lambda_{n}(M)$ share a common eigenvector (say $x$), that is,
 $$A_{\alpha}(T)x=\lambda_{n}(A_{\alpha}(T))x, ~A_{\alpha}(T_{1}\cup K_{2})x=\lambda_{n}(A_{\alpha}(T_{1}\cup K_{2}))x~~ \text{and}~ ~Mx=\lambda_{n}(M)x.$$
 Since $Mx=\lambda_{n}(M)x$ and $\lambda_{n}(M)=0$, we have $\alpha x(u)+(1-\alpha)x(v)=0$ and $\alpha x(v)+(1-\alpha)x(u)=0$, this implies that $x(u)=x(v)=0$. Since $A_{\alpha}(T_{1}\cup K_{2})x=\lambda_{n}(A_{\alpha}(T_{1}\cup K_{2}))x$ and $\lambda_{n}(A_{\alpha}(T_{1}\cup K_{2}))=2\alpha-1$, we have $(2\alpha-1)x(v)=\alpha x(v)+(1-\alpha)x(w)$, then $x(w)=0$. Let $y$ be an $(n-2)$-vector obtained from $x$ by deleting the entries $x(v),x(w)$. Since $A_{\alpha}(T_{1}\cup K_{2})x=(2\alpha-1)x$ and $x(v)=x(w)=0$, it is easy to check that $A_{\alpha}(T_{1})y=(2\alpha-1)y$. It leads to that $2\alpha-1$ is an eigenvalue of $A_{\alpha}(T_{1})$, this is contrary to the induction hypothesis. Hence $\lambda_{n}(A_{\alpha}(T))>2\alpha-1$, which completes the proof of this lemma.\hspace*{\fill}$\Box$
\end{proof}

\noindent{\bf Proof of Theorem \ref{thm4}.} Suppose that $G$ contains $k$ components, say, $G_1,\ldots,G_k$. Then $$\lambda_n(A_{\alpha}(G))=\min\{\lambda_{|V(G_{i})|}(A_{\alpha}(G_i))|i=1,\ldots,k\}.$$
Let $T_{i}$ be a spanning tree of $G_i$ for $1\leq i\leq k$. It follows from Proposition \ref{prop1} that $$\lambda_{|V(G_{i})|}(A_{\alpha}(G_i))\geq\lambda_{|V(T_{i})|}(A_{\alpha}(T_i)).$$
Now the theorem follows by Lemma \ref{le1+}.\hspace*{\fill}$\Box$

The smallest $A_{\alpha}$-eigenvalue was also studied in \cite{VN1}. In particular, the author presented the following problem: when $\alpha\in (0,1/2)$, determine how small the least $A_{\alpha}$-eigenvalue can be for a graph of given order (see \cite{VN1} Problem 30). When $1/2<\alpha<1$, it is interesting to find out which connected graph with order $n$ minimizes the smallest eigenvalue of $A_{\alpha}(G)$. Based on our numerical experiments, we propose the following conjecture.
\begin{conjecture}
Let $G$ be a connected graph of order $n$. If $\frac{1}{2}<\alpha<1$, then
$$\lambda_n(A_{\alpha}({G}))\geq \lambda_n(A_{\alpha}({K_{1,n-1}})),$$
the equality holds if and only if $G\cong K_{1,n-1}$.
\end{conjecture}

Finally, we will give an upper bound for the smallest $A_{\alpha}$-eigenvalue of a bipartite graph.

\begin{theorem}\label{thm2.3}
Let $G$ be a bipartite graph of order $n$. If $\frac{1}{2}<\alpha<1$, then
$$\lambda_n(A_{\alpha}({G}))\leq \lambda_{n}(A_{\alpha}(K_{\lceil\frac{n}{2}\rceil,\lfloor\frac{n}{2}\rfloor}))$$
with equality if and only if $G\cong K_{\lceil\frac{n}{2}\rceil,\lfloor\frac{n}{2}\rfloor}$.
\end{theorem}

\begin{proof}
Let $G$ be a bipartite graph with partitions $V_{1}$ and $V_{2}$. Suppose that $|V_{1}|=a$ and $|V_{2}|=b$. Without loss of generality, we may assume that $a\geq b$.

\noindent{\bf{Fact 1.}} If $a-b\geq 2$, then $\lambda_{n}(A_{\alpha}(K_{a,b}))<\lambda_{n}(A_{\alpha}(K_{a-1,b+1}))$.

Nikiforov \cite{VN1} showed that $\lambda_n(A_{\alpha}(K_{a,b}))=\frac{1}{2}(\alpha n-\sqrt{\alpha^2 n^2+4ab(1-2\alpha)})$. Then we have
\begin{eqnarray*}
&&\lambda_{n}(A_{\alpha}(K_{a-1,b+1}))-\lambda_{n}(A_{\alpha}(K_{a,b}))\\[3mm]
&=&\frac{1}{2}\big(\sqrt{\alpha^2 n^2+4ab(1-2\alpha)}-\sqrt{\alpha^2 n^2+4(a-1)(b+1)(1-2\alpha)} \big)\\[3mm]
&>&0,
\end{eqnarray*}
as required.

\noindent{\bf{Fact 2.}} $\lambda_{n}(A_{\alpha}(K_{\lceil\frac{n}{2}\rceil,\lfloor\frac{n}{2}\rfloor}-e))<\lambda_{n}(A_{\alpha}(K_{\lceil\frac{n}{2}\rceil,\lfloor\frac{n}{2}\rfloor}))$.

For convenience, let $H\cong K_{\lceil\frac{n}{2}\rceil,\lfloor\frac{n}{2}\rfloor}$. Suppose that $x$ is a unit eigenvector of $A_{\alpha}(H)$ corresponding to $\lambda_{n}(A_{\alpha}(H))$. Hence,
\begin{eqnarray*}
\lambda_{n}(A_{\alpha}(H))&=&\langle A_{\alpha}(H)x,x\rangle\\
&=&(2\alpha-1)\sum_{uv\in E(H)}(x(u)^{2}+x(v)^{2})+(1-\alpha)\sum_{uv\in E(H)}(x(u)+x(v))^{2}.
\end{eqnarray*}
Since $\lambda_{n}(A_{\alpha}(H))>0$, there exists an edge $e=uv$ such that $(2\alpha-1)(x(u)^{2}+x(v)^{2})+(1-\alpha)(x(u)+x(v))^{2}>0$. Therefore,
\begin{eqnarray*}
&&\lambda_n(A_{\alpha}(H-uv))-\lambda_n(A_{\alpha}(H))\\
&\leq& \langle A_{\alpha}(H-uv)x,x\rangle-\langle A_{\alpha}(H)x,x\rangle \\
&=&-(2\alpha-1)(x(u)^{2}+x(v)^{2})-(1-\alpha)(x(u)+x(v))^{2}\\
&<&0. \ \
\end{eqnarray*}
This completes the proof of Fact 2.

According to Fact 1, Fact 2 and Proposition \ref{prop1}, we can see that:
if $a-b\geq 2$, then $\lambda_{n}(G)\leq\lambda_{n}(A_{\alpha}(K_{a,b}))<\lambda_{n}(A_{\alpha}(K_{\lceil\frac{n}{2}\rceil,\lfloor\frac{n}{2}\rfloor}))$; if $a-b\leq 1$ and $G\ncong K_{\lceil\frac{n}{2}\rceil,\lfloor\frac{n}{2}\rfloor}$, then $\lambda_{n}(G)\leq\lambda_{n}(A_{\alpha}(K_{\lceil\frac{n}{2}\rceil,\lfloor\frac{n}{2}\rfloor}-e))<\lambda_{n}(A_{\alpha}(K_{\lceil\frac{n}{2}\rceil,\lfloor\frac{n}{2}\rfloor}))$. Thus we complete the proof.\hspace*{\fill}$\Box$
\end{proof}

\noindent{\textbf{Acknowledgement.}} The authors would like to express their gratitude to the anonymous reviewers for their kind suggestions on the original manuscript. The first author was supported by the National Natural Science Foundation of China (No. 11401211) and Fundamental Research Funds for the Central Universities (No. 222201714049). The third author was supported by the National Natural Science Foundation of China (No. 11471121).

\small {

}

\end{document}